\documentstyle[amssymb,amsfonts,12pt]{amsart}

\newtheorem{theorem}{Theorem}[section]
\newtheorem{lemma}[theorem]{Lemma}

\newtheorem{proposition}[theorem]{Proposition}

\newtheorem*{ack*}{Acknowledgment}

\def\R{{\mathbb R}}

\def\C{{\mathbb C}}
\def\nint{\mathop{\diagup\kern-13.0pt\int}}
\def\Z{{\mathbb Z}}

\def\A{{\mathbb A}}

\def\supp{{\operatorname{supp}}}
\def\bas{\begin{align*}}
\def\eas{\end{align*}}
\def\bi{\begin{itemize}}
\def\ei{\end{itemize}}
\newenvironment{proof}{\noindent {\bf Proof} }{\endprf\par}
\def \endprf{\hfill  {\vrule height6pt width6pt depth0pt}\medskip}
\def\emph#1{{\it #1}}

\begin{document}

\author{Ciprian Demeter}
\address{Department of Mathematics, Indiana University,  Bloomington IN}
\email{demeterc@@indiana.edu}
\author{Shaoming Guo}
\address{Department of Mathematics, Indiana University,  Bloomington IN}
\email{shaoguo@@indiana.edu}

\thanks{The first author is partially supported by the NSF grant DMS-1161752}
\title[Schr\"odinger maximal function]{Schr\"odinger maximal function estimates via the pseudoconformal transformation}

\begin{abstract}
We present an alternative way to recover the recent result from \cite{LR} using the pseudoconformal transformation.
\end{abstract}
\maketitle

\section{Introduction}
Recall that the solution of the  Schr\"odinger equation
\begin{equation}
\label{e2}
i\partial_t u(x,t)+\Delta u(x,t)=0,\;x\in\R^n,\;t\ge 0
\end{equation}
with initial data $u_0\in L^2(\R^n)$ is given by
$$u(x,t)=e^{it\Delta}u_0(x)=\int_{\R^n}\widehat{u}_0(\xi)e^{2\pi ix\cdot\xi-4\pi^2it|\xi|^2}d\xi.$$
A fundamental open question for $n\ge 2$ is identifying the smallest Sobolev index $s>0$ for which
$$\lim_{t\to 0}u(x,t)=u_0(x)\; a.e., \text{ for each }u_0\in H^s(\R^n).$$

The main goal of this note is to give an alternative argument for the following recent result of Luc\`a and Rogers, which proves a lower bound on the Sobolev regularity index $s$.
\begin{theorem}\label{main}
Let $n\ge 2$ and $s<\frac{n}{2(n+2)}$. Then there exist $R_k\to\infty$ and $f_k\in L^2(\R^n)$ with $\widehat{f_k}$ supported in the annulus $|\xi|\sim R_k$ such that
\begin{equation}
\label{e1}
\lim_{k\to\infty}\frac{\|\sup_{0<t\lesssim 1}|e^{it\Delta}f_k(x)|\|_{L^2(B(0,1))}}{R_k^{s}\|f_k\|_{L^2(\R^n)}}=\infty.
\end{equation}
\end{theorem}
\medskip

We use the pseudoconformal symmetry, according to which, if $u(x,t)$ solves \eqref{e2} then so does
$$v(x,t)=\frac1{t^{n/2}}\bar{u}(\frac{x}{t},\frac1{t})e^{i\frac{|x|^2}{4t}}.$$
Moreover, the initial data of the two solutions will have comparable $L^2$ norms. See the Appendix.
We will start with a solution $u$ (the same as the one in \cite{LR}) that is big on a cartesian set $X\times T$ of $(x,t)$. The set $X$ will be a small neighborhood of a rescaled copy of $\Z^n$ inside $[-1,1]^n$, while $T$ will be a discrete lattice inside $t\sim 1$. The measure of $X$ will be significantly smaller than 1, of order $R_k^{-\alpha_n}$, for some $\alpha_n>0$.  The property that our construction exploits is that the set $Y=\frac{X}{T}$ can be made much larger than $X$, in fact it can be made to have measure comparable to 1. Note that the new solution $v$ will now be big for each $x\in Y$ (for some $t$ depending on $x$). This will be enough to prove Theorem \ref{main}.
\medskip

Let us compare our approach with other recent ones. Luc\`a and Rogers \cite{LR} use the Galilean symmetry, according to which
if $u(x,t)$ solves \eqref{e2} then so does
$$v(x,t)=u(x-t\theta,t)e^{it\frac{|\theta|^2}{4}}e^{i\frac{x\cdot\theta}2}$$
for arbitrary $\theta\in\R^n$. Moreover, the initial data of the two solutions will have comparable $L^2$ norms. As mentioned before, they start with the same $u$, and thus have the same $X,T$. Their observation is that, for appropriate $\theta$, the set $Y=X-\theta T$ will have measure comparable to 1.
\medskip

Bourgain \cite{Bo} constructs a solution $u$ which has two attributes. On the one hand, it is big on a cartesian product $X\times \{0\}$. So $T=\{0\}$. The second property of $u$ is that it is very symmetrical, almost invariant under a certain space-time translation. More precisely, for an appropriate $\nu\in\R^{n+1}$
\begin{equation}
\label{e3}
u(x,t)\approx u((x,t)+s\nu)
\end{equation}
will hold for all $x\in B(0,1)$ and all $t,s\sim 1$. The original small set $X$ where $u$ was large gets amplified from the fact that the $x$ projection of the set $$Y=(X\times \{0\})+[\frac1{10},10]\nu$$ has measure comparable to $1$.
\medskip

In both our example and the one from \cite{LR}, the Fourier transform $\widehat{u}_0$ of the initial data is essentially the characteristic function of a small neighborhood of a rescaled (and truncated) copy of $\Z^n$. In Bourgain's construction, the mass lives on a small portion of this set, where lattice points are restricted to a sphere. The key is that the lift of this sphere to the paraboloid $(\xi,|\xi|^2)$ is a collection of points that live in a hyperplane $H\subset \R^{n+1}$. The existence of a nonzero vector $\nu\in H^{\perp}$ is what makes the remarkable symmetry \eqref{e3} possible.
\medskip

In terms of the actual mathematics that is involved in proving that the enhanced set $Y$ has measure comparable to 1, the three methods described above are at least superficially different. Luc\`a and Rogers derive a quantitative version of the ergodic theorem involving the Funk-Hecke theorem. Bourgain uses a bit of Fourier analysis but his argument also has diophantine flavor. Our argument elaborates on  a quantitative version of the multidimensional Dirichlet principle, which in its simplest form can be stated as follows.
\begin{lemma}
\label{8}
Given $y_1,\ldots,y_n\in [0,1]$ and a real number $N\ge 1$, there is $1\le p\le N+2$ such that
$$\max_{1\le i\le n}\|py_i\|\le \frac1{N^{1/n}}.$$
\end{lemma}
Here and in the following, $\|x\|$ will denote the distance of $x$ to $\Z$. The proof of this lemma is an immediate application of pigeonholing.
\medskip

It is hard to conjecture what the optimal $s$ in Theorem \ref{main} should be. The authors feel that the likeliest possibility is $s=\frac{n}{2(n+1)}$. If one runs a multilinear type Bourgain--Guth  argument for this problem (as was done in \cite{Bo}), the $n+1$ linear term has a favorable estimate consistent with this value of $s$. Another interesting question is whether the optimal $s$ is the same  for a larger class of curved hyper-surfaces $(\xi,\varphi(\xi))$ generalizing the paraboloid $(\xi,|\xi|^2)$. It is worth mentioning that Bourgain exhibits  a surface
$$\varphi(\xi)=\langle A\xi,\xi\rangle+O(|\xi|^3)$$
with $A$ positive definite, for which a stronger result is proved: Theorem \ref{main} will hold even with $s<\frac{n-1}{2n}$ ($n\ge 3$) and $s<\frac{5}{16}$ ($n=2$).

\medskip

\begin{ack*}
The authors thank J. Bennett, R. Luc\`a, K. Rogers and A. Vargas for a few interesting discussions on this topic.
\end{ack*}

\section{The main construction and the proof of Theorem \ref{main}}
Via rescaling, Theorem \ref{main} will follow from the following result.
\begin{theorem}\label{main1}
Let $n\ge 2$ and $s<\frac{n}{2(n+2)}$. Then there exist $R_k\to\infty$ and $v_k\in L^2(\R^n)$ with $\widehat{v_k}$ supported in the annulus $|\xi|\sim 1$ such that
\begin{equation}
\label{e25}
\lim_{k\to\infty}\frac{\|\sup_{0<t\lesssim R_k}|e^{it\Delta}v_k(x)|\|_{L^2(B(0,R_k))}}{R_k^{s}\|v_k\|_{L^2(\R^n)}}=\infty.
\end{equation}
\end{theorem}
We will prove this in the end of the section, using some elementary number theoretical results derived in Section \ref{50}.
\medskip

For $0<u<v$ define the annuli
$$\A_{u,v}=\{x\in\R^n:\;u<|x|<v\}.$$
Fix $\sigma<\frac1{n+2}$. Fix a Schwartz function $\theta$ on $\R^n$ whose Fourier transform is supported inside $\A_{4^{-n-3},4\sqrt{n}}$ and equals 1 on $\A_{4^{-n-2},2\sqrt{n}}$. The next three lemmas are used to align the phases of an exponential sum so that the absolute value of the sum is comparable to the number of exponentials in the sum.
\begin{lemma}
\label{4}
There exists $\epsilon_1>0$ so that for each $R$ large enough, the following holds:

For each $x\in \A_{4^{-n-2},2\sqrt{n}}$, each $t\in (0,1)$ and each $\xi'\in\R^n$ with $|\xi'|\le \epsilon_1R$ we have
$$|\int e^{2\pi i[(x-\frac{2t\xi'}{R})\cdot \xi-\frac{t}{R}|\xi|^2]}\theta(\xi)d\xi-1|<\frac12.$$
\end{lemma}
\begin{proof}
Let $$\psi(\xi)=-2\pi [\frac{2t\xi'}{R}\cdot \xi+\frac{t}{R}|\xi|^2].$$
Use the fact that $$\int e^{2\pi ix\cdot \xi}\theta(\xi)d\xi=1.$$
Then estimate $$|\int  e^{2\pi ix\cdot \xi}\theta(\xi)[e^{i\psi(\xi)}-1]d\xi|\le 2\int_{|\xi|>C}|\theta(\xi)|d\xi+2\sup_{|\xi|\le C}|\psi(\xi)|\int_{\R^n}|\theta|.$$
Choose first $C$ so that $$\int_{|\xi|>C}|\theta(\xi)|d\xi<\frac14.$$
Then note that $$\sup_{|\xi|\le C}|\psi(\xi)|\le 2\pi(\frac{C^2}{R}+2\epsilon_1C).$$
Choose $\epsilon_1$ so small that $$4\pi(\frac{C^2}{R}+2\epsilon_1C)\int_{\R^n}|\theta|<\frac14.$$
for all $R$ large enough.
\end{proof}
The following lemma is rather trivial.
\begin{lemma}
\label{5}
Let $\Omega$ be a finite set.
Consider $a_{\xi'},b_{\xi'}\in\C$ for $\xi'\in\Omega$ such that
$$\max|a_{\xi'}-1|\le \delta_1$$
$$\max|b_{\xi'}-1|\le \delta_2.$$
Then
$$|\sum_{\xi'\in \Omega}a_{\xi'}b_{\xi'}-|\Omega||\le |\Omega|(\delta_1\max|b_{\xi'}|+\delta_2\max|a_{\xi'}|+\delta_1\delta_2).$$
\end{lemma}

For $\epsilon_2>0$ small enough (depending only on $\theta$, as revealed in the proof of Proposition \ref{40}), define
$$X:=(R^{\sigma-1}\Z^n+B(0,\frac{\epsilon_2}R))\cap \A_{4^{-n-2},2\sqrt{n}}$$
$$T:=(R^{2\sigma-1}\Z)\cap (4^{-n-1},1),$$
and
$$\Omega:=(R^{1-\sigma}\Z^n)\cap B(0,\epsilon_1R).$$
Define also the Fourier transform of the initial data
$$\widehat{u_0}(\xi)=\sum_{\xi'\in\Omega}\theta(\xi-\xi').$$
Note that
\begin{equation}
\label{24}
\|u_0\|_2\sim R^{\frac{\sigma n}{2}}.
\end{equation}
and
\begin{equation}
\label{41}
\supp\; u_0\subset \A_{4^{-n-3},4\sqrt{n}}.
\end{equation}
The following is essentially proved in (3.2) from \cite{LR}.
\begin{proposition}
\label{40}
We have the following lower bound for each $(x,t)\in X\times T$
\begin{equation}
\label{20}
|e^{i\frac{t}{2\pi R}\Delta}u_0(x)|\gtrsim R^{\sigma n}.
\end{equation}
\end{proposition}
\begin{proof}
Note first that
$$e^{i\frac{t}{2\pi R}\Delta}u_0(x)=\sum_{\xi'\in\Omega}e^{2\pi i[x\cdot \xi'-\frac{t}{R}|\xi'|^2]}\int e^{2\pi i[(x-\frac{2t\xi'}{R})\cdot \xi-\frac{t}{R}|\xi|^2]}\theta(\xi)d\xi$$
One easily checks that for $(x,t)\in X\times T$ and $\xi'\in\Omega$ we have
$$x\cdot \xi'-\frac{t}{R}|\xi'|^2\in \Z+B(0,\epsilon_3),$$
where $\epsilon_3$ can be chosen as small as desired by choosing $\epsilon_2$ small enough. In particular, we can make sure that
$$|e^{2\pi i[x\cdot \xi'-\frac{t}{R}|\xi'|^2]}-1|<\min\{\frac1{100}, \frac{1}{100}\int|\theta|\}.$$
It suffices now to combine this with Lemma \ref{4} and Lemma \ref{5}, once we also note that $$|\Omega|\sim R^{\sigma n}.$$
\end{proof}
\medskip
Recall that $u_0$ depends on $R$, so we might as well write $u_0=u_{0,R}$.
Let now $u_{0, R}(x,t)=e^{it\Delta}u_{0,R}(x)$ and let
$$v_R(x,t)=\frac1{t^{n/2}}\bar{u}_R(\frac{x}{t},\frac1{t})e^{i\frac{|x|^2}{4t}}$$
be its pseudoconformal transformation. The proposition in the Appendix shows that $v_R$ solves the Schr\"odinger equation with some initial data that we call $v_{0,R}$.

We record the properties of $v_R$ in the following proposition.
\begin{proposition}
\label{11}
We have for each large enough $R$ such that $R^{\sigma}$ is an integer
\begin{equation}
|\{x\in B(0,R):\sup_{0<t\lesssim R}|v_R(x,t)|\gtrsim R^{\sigma n-\frac{n}2}\}|\gtrsim R^n
\end{equation}
\begin{equation}
\|v_{0,R}\|_2\sim R^{\frac{\sigma n}2}
\end{equation}
\begin{equation}
\frac{\|\sup_{0<t\lesssim R}|v_R(x,t)|\|_{L^2(B_R)}}{\|v_{0,R}\|_2}\gtrsim R^{\frac{\sigma n}2}
\end{equation}
\begin{equation}
\supp \;\widehat{v_{0,R}}\subset 4\pi (\supp \;u_{0,R})\subset \A_{4^{-n-2}\pi, 16\pi\sqrt{n}}.
\end{equation}
\end{proposition}
\begin{proof}
The first property follows from \eqref{20} and \eqref{21}. The second one follows from \eqref{24} and \eqref{25}. The third one is a consequence of the first two. The fourth one also follows from \eqref{25} and \eqref{41}.
\end{proof}
\bigskip

Let now $s<\frac{n}{2(n+2)}$. The proof of Theorem \ref{main1} for this $s$ will now immediately follow by choosing $\sigma<\frac1{n+2}$ such that $\frac{\sigma n}{2}>s$, and by using $v_k=v_{0, R_k}$ from Proposition \ref{11}, with $R_k^\sigma$ an integer that grows to infinity with $k$.

\section{Number theoretical considerations}
\label{50}
\begin{lemma}
\label{7}
For each  large enough real number $N$ there is $S=S_N\subset [0,1]^n$ with $|S|\ge \frac34$ such that for each $(y_1,\ldots,y_n)\in S$ there exists $p\in [4^{-n-1}N,N+2]$ satisfying
\begin{equation}
\label{9}
\max_{1\le i\le n}\|py_i\|\le \frac1{N^{\frac1n}}.
\end{equation}
\end{lemma}
\begin{proof}
Using Lemma \ref{8}, we know that \eqref{9} holds for each $(y_1,\ldots,y_n)\in [0,1]^n$, if we allow $p\in [1,N+2]$. We need an upper bound for those $(y_1,\ldots,y_n)$ corresponding to $1\le p\le 4^{-n-1}N$. For each $p$ define
$$A_p=\{(y_1,\ldots,y_n)\in [0,1]^n:\;\max_{1\le i\le n}\|py_i\|\le \frac1{N^{\frac1n}}\}=$$
$$=\bigcup_{0\le k_i\le p}\{(y_1,\ldots,y_n)\in [0,1]^n:\;\max_{1\le i\le n}|py_i-k_i|\le \frac1{N^{\frac1n}}\}.$$
The crude estimate
$$|A_p|\le (p+1)^n(\frac{2}{N^{\frac1n}p})^n<\frac{4^n}{N}$$
leads to
$$|\bigcup_{1\le p\le 4^{-n-1}N}A_p|\le \frac14.$$
\end{proof}
\begin{proposition}
Assume $R^{\sigma}$ is a large enough  integer.
Let $$\frac{XR}{T}=\{\frac{xR}{t}:x\in X, t\in T\}.$$
Then
\begin{equation}
\label{21}
|\frac{XR}{T}\cap B(0,R)|\gtrsim R^n.
\end{equation}
\end{proposition}
\begin{proof}
It suffices to prove  that
$$|\frac{X}{T}\cap [0,1]^n|\ge \frac12.$$
This can be written as $|U|\ge \frac12$ where
$$U=\{x\in[0,1]^n:\max_{1\le i\le n}|x_i-R^{-\sigma}\frac{k_i}{p}|\le \frac{\epsilon_2}R,\text{ with }$$$$4^{-n-2}R^{1-\sigma}\le |k|\le 2\sqrt{n}R^{1-\sigma},\;4^{-n-1}R^{1-2\sigma}\le p\le R^{1-2\sigma} \}.$$
This will follow if we prove that the larger set (we have dropped the restriction on $k_i$)
$$V=\{x\in[0,1]^n:\max_{1\le i\le n}|x_i-R^{-\sigma}\frac{k_i}{p}|\le \frac{\epsilon_2}R,\text{ with }4^{-n-1}R^{1-2\sigma}\le p\le R^{1-2\sigma} \}$$
satisfies
$|V|\ge \frac34.$ Indeed, note first that the restriction
$$|k|\le 2\sqrt{n}R^{1-\sigma}$$is redundant, as the inequality $\max|k_i|\le R^{1-\sigma}+1\le 2R^{1-\sigma}$ is forced by the combination of $x\in[0,1]^n$, $\max_{1\le i\le n}|x_i-R^{-\sigma}\frac{k_i}{p}|\le \frac{\epsilon_2}R$ and $p\le R^{1-2\sigma}$.
Second, note that
$$W=\{x\in[0,1]^n:\max_{1\le i\le n}|x_i-R^{-\sigma}\frac{k_i}{p}|\le \frac{\epsilon_2}R,\text{ with }$$$$|k|\le 4^{-n-2}R^{1-\sigma},\;4^{-n-1}R^{1-2\sigma}\le p\le R^{1-2\sigma} \}$$
satisfies  $W\subset[0,\frac12]^n$ and thus $|W|<\frac14$.
\medskip

We next focus on showing that $|V|\ge \frac34$. The inequality $$\max_{1\le i\le n}|x_i-R^{-\sigma}\frac{k_i}{p}|\le \frac{\epsilon_2}R$$
can be written as
$$\max_{1\le i\le n}\|p\{R^{\sigma}x_i\}\|\le \frac{\epsilon_2pR^{\sigma}}R,$$
where $\{z\}$ is the fractional part of $z$. Note that given the lower bound  $p\ge 4^{-n-1}R^{1-2\sigma}$ and that $\sigma<\frac1{n+2}$, for $R$ large enough we have
$$\frac{\epsilon_2pR^{\sigma}}R\ge \frac1{(R^{1-2\sigma}-2)^{\frac1n}}.$$
Let $N:=R^{1-2\sigma}-2$. Then $V_0\subset V$ where
$$V_0=\{x\in[0,1]^n:\max_{1\le i\le n}\|p\{R^{\sigma}x_i\}\|\le \frac1{N^{\frac1n}},\text{ for some }4^{-n-1}N\le p\le N+2 \}.$$
Since $R^{\sigma}$ is an integer, the map
$$(x_1,\ldots,x_n)\mapsto (\{R^{\sigma}x_1\},\ldots, \{R^{\sigma}x_n\})$$
is a measure preserving transformation on $[0,1]^n$. The fact that $|V_0|\ge \frac34$ now follows from Lemma \ref{7}.
\end{proof}

\section{Appendix}
We record below the following classical result. See for example \cite{Ta}, page 72.
\begin{proposition}
If $u(x,t)$ solves \eqref{e2} then so does its pseudoconformal transformation
$$v(x,t)=\frac1{t^{n/2}}\bar{u}(\frac{x}{t},\frac1{t})e^{i\frac{|x|^2}{4t}}.$$
Moreover, the initial data $u_0(x)=u(x,0)$, $v_0(x)=v(x,t)$ are related by the formula
\begin{equation}
\label{25}
\widehat{v_0}(y)=Cu_0(4\pi y).
\end{equation}
In particular, $$\|u_0\|_2\sim\|v_0\|_2.$$
\end{proposition}

\end{document}